\documentclass[conference,a4paper]{IEEEtran}
\ifCLASSINFOpdf
  \usepackage[pdftex]{graphicx}
\else
\fi

\usepackage[cmex10]{amsmath}
\usepackage{url}

\usepackage{amsthm}
\usepackage{paralist}

\newtheorem{theorem}{Theorem}
\theoremstyle{remark}
\usepackage{tikz}
\usetikzlibrary{arrows}

\begin{document}
%
\title{Optimal control of storage for arbitrage,\\ with applications to
  energy systems}

\author{
\IEEEauthorblockN{J.R. Cruise}
\IEEEauthorblockA{Actuarial Mathematics and Statistics\\Heriot-Watt University\\Edinburgh, UK\\
Email: r.cruise@hw.ac.uk}
\and
\IEEEauthorblockN{R.J. Gibbens}
\IEEEauthorblockA{Computer Laboratory\\University of Cambridge\\
Cambridge, UK\\
Email: richard.gibbens@cl.cam.ac.uk}
\and
\IEEEauthorblockN{S. Zachary}
\IEEEauthorblockA{Actuarial Mathematics and Statistics\\Heriot-Watt University\\
Edinburgh, UK\\
Email: s.zachary@hw.ac.uk}
}


%


\maketitle

\begin{abstract}
  We study the optimal control of storage which is used for arbitrage,
  i.e.\ for buying a commodity when it is cheap and selling it when it
  is expensive.  Our particular concern is with the management of
  energy systems, although the results are generally applicable.  We
  consider a model which may account for nonlinear cost functions,
  market impact, input and output rate constraints and
  inefficiencies or losses in the storage process.  We develop an
  algorithm which is maximally efficient in the sense that it
  incorporates the result that, at each point in time, the optimal
  management decision depends only a finite, and typically short, time
  horizon.  We give examples related to the management of a real-world
  system.
\end{abstract}


\section{Introduction}
\label{sec:introduction}

How should one optimally control a store which is used to make money
by buying a commodity when it is cheap, and selling it when it is
expensive?  We are interested in this question primarily in the
context of electrical energy systems, where a store may, for example,
take in energy at night, when there may be a surplus of supply over
demand rendering the excess energy cheap, and release that energy
during the day.  However, the mathematics we develop is of course more
generally applicable.

A major constraint on the operation of electrical energy systems---for
example, the UK national grid or similar systems in other countries or
continents---is that supply and demand need to be kept in very close
balance at all times.  It has always been the case that electricity
demand is highly variable, notably on daily, weekly and annual cycles,
although this variation is in general at least predictable.  However,
the increasing reliance on renewable sources of generation such as
wind and solar power is now introducing both variability and
unpredictability in electricity supplies.  In order to assist in
keeping supply and demand well balanced it is useful to be able to
shift electrical energy through time.  The most obvious way to do this
is through \emph{storage}, which rearranges the profile in time of
energy \emph{supply}.  However, the profile in time of \emph{demand}
may also be rearranged through what is generally referred to as
\emph{demand-side management}, and it should be noted that the
postponement of demand is mathematically equivalent to the use of
negative storage---although the practical difficulties with
demand-side management are somewhat different.

Storage may assist in a large number of ways, most notably:
\begin{compactenum}[(i)]
\item in shifting energy from times of low demand, when its
  generation is typically cheap, to times of high demand, when its
  generation is typically expensive;
\item in stabilising the system with respect to small and transient
  imbalances;
\item in reacting to major disturbances, such as sudden loss of
  generation, transmission failures, or sudden surges in demand.
\end{compactenum}

Our interest here is primarily in the first of these---see, for
example, \cite{BGK,GTL,HMN,WW} for a broader discussion of storage.
We take an economic view, and investigate the value of energy storage
for \emph{arbitrage}, that is smoothing price fluctuations over time.
Thus it is assumed that energy is always available from somewhere, at
a sufficiently high price, and that the value of storage consists of
its ability to buy energy when it is cheap and release it when it is
expensive.  We work here in a deterministic setting in which we assume
that all relevant buying and selling prices are known in advance.

We think of the available storage as a single store.  Its \emph{value}
is equal to the profit which can be made by a notional store ``owner''
buying and selling as above.  In the case where the activities of the
store are sufficiently significant as to have a market impact (the
store becomes a ``price maker''), the \emph{system} or \emph{societal}
value of the store may be similarly calculated by adjusting the
notional buying and selling prices so that the store ``owner'' is
required to bear also the external costs of the store's activities.
Thus our framework is in this respect completely general.  We also
allow for nonlinear cost functions, for differences in  buying and
selling prices, for inefficiencies in the storage process, and for
input and output rate constraints.

In Section~\ref{sec:problem} we formally define the relevant
mathematical problem and characterise mathematically its optimal
solution.  In Section~\ref{sec:algorithm} we provide an algorithm for
its solution, which is efficient in the sense we there explain.  In
particular the decisions to be made at each point in time typically
depend only on a very short future horizon---which is identifiable,
but not determined in advance.  In Section~\ref{sec:example} we give
examples based on real data and a real pumped storage facility.
Finally, in Section~\ref{sec:conclusions} we outline some extensions
and give concluding remarks.


\section{Problem formulation and characterisation of solution}
\label{sec:problem}

We work in discrete time, which we take to be integer.  We assume that
the store has total \emph{capacity} of $E$ units of energy, and input
and output \emph{rate constraints} of $P_i$ and $P_o$ units of power
respectively.  We consider also a time-independent
(in)efficiency~$\eta$ associated with the store.  This may be defined
as the fraction of energy input which is available for output.  It may
be captured in our model either by adjusting buy prices by a factor
$1/\eta$ or by multiplying sell prices by $\eta$ (the values of $E$ in
the two cases differing by a factor of $\eta$).  Hence, without loss
of generality, we take $\eta=1$ throughout in our mathematical
formulation of the problem and its solution.  (A further type of
(in)efficiency, which may be regarded as \emph{leakage} over time is
considered in Section~\ref{sec:conclusions}.)

Both buying
and selling prices at time $t$ may conveniently be represented by a
function $C_t(x)$ with $C_t(0)=0$, which is increasing and convex in
$x$ and which, for positive $x$, is the price of buying $x$ units of
energy, and, for negative $x$, is the negative of the price for
selling $-x$ units of energy.  Thus the cost of increasing the level
of energy in the store by $x$, positive or negative, is always
$C_t(x)$.  The convexity assumption corresponds, for each time $t$, to
an increasing cost to the store of buying each additional unit of
energy, a decreasing revenue obtained for selling each additional unit
of energy, and every unit buying price being at least as great as
every unit selling price.

As indicated above, if the problem is to determine the value of the
store to the entire energy system, or to society, then these prices
are taken to be those appropriate to the system or societal costs.
Thus, for example, for $x$ positive, $C_t(x)$ is the price paid by the
store at time $t$ for $x$ units of energy plus the increased cost paid
by other energy users at that time as a result of the store's purchase
increasing market prices.

A special case is that of a ``small'' store, whose operations do not
influence the market (the store is a ``price-taker'' rather than a
``price-maker''), and which at time $t$ buys and sells energy at given
prices per unit of $c^{(b)}_t$ and $c^{(s)}_t$ respectively, where we assume
that $c^{(b)}_t\ge c^{(s)}_t$.  Here the function $C_t(x)$ is given by
\begin{equation}
  \label{eq:1}
  C_t(x) =
  \begin{cases}
    c^{(b)}_t x & \quad\text{if $x\ge0$}\\
    c^{(s)}_t x & \quad\text{if $x<0$}.
  \end{cases}
\end{equation}

Finally, we assume
that all prices are known in
advance, so that the problem of controlling the store is
deterministic.

Consider the problem of controlling the store so as to maximise profit
over a time interval $[0,T]$.  Note that  the rate constraints may, if
we choose, be absorbed into the cost function---for example, by
defining, for each $t$
\begin{displaymath}
  C_t(x) =
  \begin{cases}
    C_t(-P_o) & \text{for $x\le-P_o$}\\
    \infty & \text{for $x>P_i$}
  \end{cases}
\end{displaymath}
(formally this defines an extension of the range of the cost functions
to include the point at $\infty$ but it is readily verified that this
causes no problems).  For simplicity in the presentation of the theory
below we assume this to have been done; thus see
Figure~\ref{fig:cost_fn} for an illustration of a typical cost function.
\begin{figure}[!t]
\centering
\begin{tikzpicture}[scale=0.75]
  \draw[->,>=stealth'] (-5,0) -- (5,0) node [right] {$x$};
  \draw[->,>=stealth'] (0,-1) -- (0,5) node [left] {$C_t(x)$};
  
  \draw[very thick] (-5,-1) -- (-4,-1) .. controls (-3,-1) .. (0,0) .. controls (4,3) .. (4,4) -- (4,5);
  \draw[dashed, thick] (-4,0) node [above] {$-P_0$} -- (-4,-1);
  \draw[dashed, thick] (4,0) node [below] {$P_i$} -- (4,4);
  \draw (-2,-1) node [below] {sell};
  \draw (3,0.75) node [above] {buy};
 \end{tikzpicture}
\caption{Illustrative cost function $C_t$, incorporating also rate constraints.}
\label{fig:cost_fn}
\end{figure}
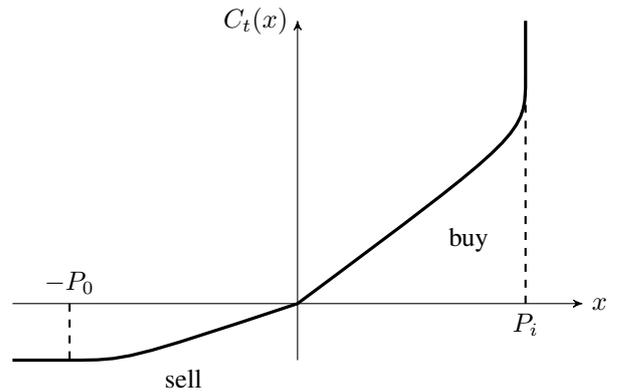

Denote the successive levels of the store by a vector
$S=(S_0,\dots,S_T)$ where $S_t$ is the level of the store at each
successive time $t$.  Define also the vector
$x(S)=(x_1(S),\dots,x_T(S))$ by $x_t(S)=S_t-S_{t-1}$ for each
$t\ge1$,
so that $x_t(S)$ represents the
energy added to the store at time~$t$.  It is convenient to assume
that both the initial level $S_0$ and the final level $S_T$ of the
store are fixed in advance at $S_0=S^*_0$ and $S_T=S^*_T$.  (If the
final level $S_T$ is not fixed and the cost function $C_T$ is strictly
increasing, then, for an optimal control, we may take $S_T$ to be
minimised; however, we might, for example, require $S_T=S_0$---as a
contribution to a toroidal solution.)


The problem thus becomes:
\begin{compactitem}
\item[$\mathbf{P}$:]
  choose $S$ so as to minimise
  \begin{equation}
    \label{eq:2}
    G(S) := \sum_{t=1}^T C_t(x_t(S))
  \end{equation}
  with $S_0=S^*_0$ and $S_T=S^*_T$, and subject to the capacity
  constraints
  \begin{equation}
    \label{eq:3}
    0 \le S_t\le E,
    \qquad 1 \le t \le T-1.
  \end{equation}
\end{compactitem}

In the case where the cost functions $C_t$ are linear, or piecewise
linear, as in the ``small store'' case given by \eqref{eq:1}, and in
which rate constraints may be present, the problem~$\mathbf{P}$ may be
reformulated as a linear programming problem, and solved by, for
example, the use of the minimum cost circulation algorithm (see, for
example, \cite{Boyd,AMO}).  Our aim in the present paper is to deal with
the general case, and to develop an algorithm which proceeds locally
in time, providing a solution which is efficient both in the general
and in the linear case.


In Theorem~\ref{thm:ls} below we use strong Lagrangian
theory~\cite{Boyd,Whi} to give sufficient conditions for a value $S^*$
of $S$ to solve the problem~$\mathbf{P}$.  We then discuss briefly how
Lagrangian theory may be used to show that a solution of the form
given always exists.  The truth of this latter assertion is also
demonstrated in Section~\ref{sec:algorithm} when we consider an
algorithm for the determination of the solution.

\begin{theorem}
  \label{thm:ls}
  Suppose that there exists a
  vector $\mu^*=(\mu^*_1,\dots,\mu^*_T)$ and a value
  $S^*=(S^*_0,\dots,S^*_T)$ of $S$ such that
  \begin{compactenum}[(i)]
  \item $S^*$ is feasible for the stated problem,
  \item for each $t$ with $1\le t\le T$, $x_t(S^*)$ minimises
    $C_t(x)-\mu^*_tx$ over all $x$,
  \item the pair $(S^*,\mu^*)$ satisfies the complementary slackness
    conditions, for $1\le t\le T-1$,
    \begin{equation}
      \label{eq:4}
      \begin{cases}
        \mu^*_{t+1} = \mu^*_t & \quad\text{if $0 < S^*_t < E$,}\\
        \mu^*_{t+1} \le \mu^*_t & \quad\text{if $S^*_t = 0$,}\\
        \mu^*_{t+1} \ge \mu^*_t & \quad\text{if $S^*_t = E$.}
      \end{cases}
    \end{equation}
  \end{compactenum}
  Then $S^*$ solves the stated problem~$\mathbf{P}$.
\end{theorem}

\begin{proof}
  Let $S$ be any vector which is feasible for the problem (with
  $S_0=S^*_0$ and $S_T=S^*_T$).  Then, from
  the condition~(ii),
  \begin{displaymath}
    \sum_{t=1}^T\left[C_t(x_t(S^*)) - \mu^*_t x_t(S^*)\right]
    \le
    \sum_{t=1}^T\left[C_t(x_t(S)) - \mu^*_t x_t(S)\right].
  \end{displaymath}
  Rearranging and recalling that $S$ and $S^*$ agree at $0$ and at
  $T$, we have
  \begin{equation}
  \begin{split}
   &\sum_{t=1}^T C_t(x_t(S^*)) -  \sum_{t=1}^T C_t(x_t(S)) \\
    & \qquad \le \sum_{t=1}^T \mu^*_t (S^*_t -  S^*_{t-1} - S_t +
     S_{t-1})\\
    & \qquad = \sum_{t=1}^{T-1}(S^*_t - S_t)(\mu^*_t - \mu^*_{t+1})\\
   & \qquad \le 0,
  \end{split}
\end{equation}
  by the condition~(iii), so that the result follows.  
\end{proof}

As previously discussed, Theorem~\ref{thm:ls} gives a sufficient
condition for a pair~$(S^*,\mu^*)$ to solve the problem~$\mathbf{P}$.
As is clear from the condition~(ii) of the theorem, the vector $\mu^*$
has the interpretation that, at each time~$t$, the quantity $\mu^*_t$
may be regarded as a notional value per unit of energy in store, and
may be used to determine how much further energy to buy or sell at
that time.  An optimal solution to the problem~$\mathbf{P}$ is given
by keeping this reference value $\mu^*_t$ as constant as possible over
time (for otherwise a ``solution'' may be improved by using a more
consistent value of the vector~$\mu^*$); the exceptions occur at the
boundaries~$0$ and $C$ of the capacity constraint region, where the
above improvements may not be possible and where $\mu_t$ is allowed to
decrease immediately subsequent to those times when the store is empty
and to increase immediately subsequent to those times when it is full.

An examination of the relevant strong Lagrangian theory (again see
\cite{Boyd,Whi}) shows that the vector $\mu^*$ has a representation as
\begin{equation*}
  \mu^*_t=\sum_{u=t}^T(\alpha^*_u+\beta^*_u),
  \qquad 1 \le t \le T,
\end{equation*}
where each $\alpha^*_t$ and $\beta^*_t$ are (strong) Lagrange
multipliers associated with respectively the lower bound $0$ and upper
bound~$E$ of the capacity constraint at time $t$.  The standard
convexity condition of the supporting hyperplane theorem shows that
here a sufficient condition for the existence of such Lagrange
multipliers is given by the assumed convexity of the cost
functions~$C_t$, and this in its turn is sufficient for the existence
of a pair~$(S^*,\mu^*)$ as in Theorem~\ref{thm:ls}.  We do not give a
formal proof of this assertion here; rather the algorithm given in the
following section constructs such a pair~$(S^*,\mu^*)$ directly.

\section{Algorithm}
\label{sec:algorithm}

We now give an explicit construction of a pair $(S^*,\mu^*)$ as in
Theorem~\ref{thm:ls}, and hence also an algorithm for the
solution of the problem.
This algorithm below may briefly be described as that of attempting to
choose $(S^*,\mu^*)$ so as to satisfy the conditions of
Theorem~\ref{thm:ls}, by choosing the components of these vectors
successively in time and by keeping $\mu^*_t$ as constant as possible
over $t$, changes only being allowed at times when the store is either
empty or full.

For further simplicity, we suppose first that the cost functions $C_t$
are all strictly convex.  For any $t$ such that $1\le t\le T$ and any
(scalar) $\mu$, define $x^*_t(\mu)$ to be the unique value of $x$
which minimises $C_t(x)-\mu x$.
Note that $x^*_t(\mu)$ is then continuous and increasing (though not
necessarily strictly so) in $\mu$.  Define a sequence of times
$0=T_0<T_1<\dots<T_k=T$ and the pair $(S^*,\mu^*)$ inductively as
follows.  Suppose that $i\ge0$ is such that $T_0,\dots,T_i$ together
with $S^*_0,\dots,S^*_{T_i}$ and $\mu^*_1,\dots,\mu^*_{T_i}$, are all
defined. For each (scalar) $\mu$, define a vector
$S(\mu)=(S_1(\mu),\dots,S_T(\mu))$ by
\begin{equation}\label{eq:15}
  S_t(\mu) =
  \begin{cases}
    S^*_t,& \qquad 1\le t\le T_i\\
    S_{t-1}(\mu)+x^*_t(\mu),& \qquad T_i+1\le t\le T.
  \end{cases}
\end{equation}
Define the sets
\begin{displaymath}
  \begin{split}
    M_i = \{\mu:\exists\ T'(\mu) \text{ with $T_i+1\le T'(\mu)\le T$}\\
      \text{such that $0\le S_t(\mu)\le E$ for $T_i+1\le t<T'(\mu)$}\\
    \text{and \emph{either} $S_{T'(\mu)}(\mu)<0$ \emph{or}
      $T'(\mu)=T$, $S_{T}(\mu)<S^*_T$}
    \}.
  \end{split}
\end{displaymath}
and
\begin{displaymath}
  \begin{split}
    M_i' = \{\mu:\exists\ T'(\mu) \text{ with $T_i+1\le T'(\mu)\le T$}\\
      \text{such that $0\le S_t(\mu)\le E$ for $T_i+1\le t<T'(\mu)$}\\
    \text{and \emph{either} $S_{T'(\mu)}(\mu)>E$ \emph{or}
      $T'(\mu)=T$, $S_{T}(\mu)>S^*_T$}
    \}.
  \end{split}
\end{displaymath}
Thus $M_i$ and $M_i'$ are the sets of $\mu$ for which $S(\mu)$
violates one of the capacity constraints and first does so
respectively below or above---in either case at a time which we denote
by $T'(\mu)$.  Note that, since each $x^*_t(\mu)$ is increasing in
$\mu$, we have $\mu<\mu'$ for all $\mu\in M_i$, $\mu'\in M_i'$.  In
particular the sets $M_i$ and $M_i'$ are disjoint.  Note also that
since, for all $t$, we have $x^*_t(\mu)\to-\infty$ as $\mu\downarrow0$,
the set $M_i$ is nonempty.  Let $\bar\mu_i=\sup M_i$.  We now consider
the behaviour of $S(\bar\mu_i)$, for which there are three
possibilities:
\begin{compactenum}[(a)]
\item the vector $S(\bar\mu_i)$ is feasible; in this case we take
  $T_{i+1}=T$ and $S^*_t=S_t(\bar\mu_i)$ with $\mu^*_t=\bar\mu_i$ for
  $T_i+1\le t\le T$ (thus also $S^*_t=S_t(\bar\mu_i)$ for all $t$);
\item the vector $\bar\mu_i$ belongs to the set $M_i$; here there
  necessarily exists at least one $t<T'(\bar\mu_i)$ such that
  $S_t(\bar\mu_i)=E$ (for otherwise, by the continuity of each
  $S_t(\mu)$ in $\mu$, $\mu$ could be increased above $\bar\mu_i$
  while still belonging to the set $M_i$); define $T_{i+1}$ to be any
  such $t$, say the largest, and (again) take $S^*_t=S_t(\bar\mu_i)$
  and $\mu^*_t=\bar\mu_i$ for all $t$ such that $T_i+1\le t\le
  T_{i+1}$; note also that we then have $\bar\mu_i\in M_{i+1}$ so that
  we shall necessarily have $\bar\mu_{i+1}\ge\bar\mu_i$;
\item the vector $\bar\mu_i$ belongs to the set $M_i'$; here,
  similarly to the case (b), there necessarily exists at least one
  $t<T'(\bar\mu_i)$ such that $S_t(\bar\mu_i)=0$; define $T_{i+1}$ to
  be any such $t$, again say the largest, and again take
  $S^*_t=S_t(\bar\mu_i)$ and $\mu^*_t=\bar\mu_i$ for all $t$ such that
  $T_i+1\le t\le T_{i+1}$; further, in this case we have
  $\bar\mu_i\notin M_{i+1}$ so that we shall necessarily have
  $\bar\mu_{i+1}\le\bar\mu_i$.
\end{compactenum}

In the case where the cost functions $C_t$ are all strictly convex, it
now follows immediately from the above construction of the pair
$(S^*,\mu^*)$ that this pair satisfies the conditions (i)--(iii) of
Theorem~\ref{thm:ls}.

In the case where, for at least some $t$, the cost function $C_t$ is
convex, but not necessarily strictly convex, a little extra care is
required.  Here, for such $t$, the function $\mu\rightarrow
x^*_t(\mu)$ is not in general uniquely defined, and, for any given
choice, this function is not in general continuous.  However, in
essence, the above construction of $(S^*,\mu^*)$ continues to
hold---it is simply a matter, where necessary, of choosing the right
value of $x^*_t(\mu)$.


We summarise our results in Theorem~\ref{thm:const} below.

\begin{theorem}
  \label{thm:const}
  The pair $(S^*,\mu^*)$ given by the above recursive construction
  satisfies the conditions (i)--(iii) of Theorem~\ref{thm:ls}.
\end{theorem}



The above algorithm requires the determination, at each of the
successive times~$T_i$, $0\le i\le k-1$, of the succeeding time
$T_{i+1}$ and of the common value $\bar\mu_i$ of $\mu^*_t$ for
$T_i+1\le t\le T_{i+1}$.  This is done by looking ahead for the
minimum time horizon necessary for the above determination; the
process then restarts at the time $T_{i+1}$.  A lengthening of the
total time $T$ over which the optimization is to be performed does not
in general change the values of the times $T_i$, but rather simply
creates more of them.  In this sense both the solution to the
problem~$\mathbf{P}$ and the above algorithm are local in time, so
that the solution to $\mathbf{P}$ involves computation which grows
essentially linearly in $T$.  The typical length of the intervals
between the successive times $T_i$ depends on the shape of the cost
functions $C_t$ (notably the difference between buying and selling
prices), together with the rate at which these functions fluctuate in
time.  This is to be expected as the store operates by selling at
prices above those at which it bought, and what is important is the
frequency with which such events can occur.  For example, such
fluctuations may occur in a 24-hour cycle, and, depending on the shape
of the cost functions, the typical length of the intervals between the
successive times $T_i$ may then be of the order of around 12 hours.
These points are illustrated further in the examples of the following
section.

We observe also that, for each time $T_i$ as above, the determination
of the succeeding time $T_{i+1}$ and of $\bar\mu_i$ involves some form
of search over an interval of the real line and as such may typically
only be carried out to a specified degree of precision.  This is
inevitable given general convex cost functions.


\section{The ``small'' store}
\label{sec:example}

In this section we look further at the case of a ``small'' store,
whose operations do not influence the market, and which at time $t$
buys and sells energy at given prices per unit of $c^{(b)}_t$ and
$c^{(s)}_t$ respectively (with $c^{(b)}_t\ge c^{(s)}_t$), so that each
of the functions~$C_t(x)$ is as given by~\eqref{eq:1}.  We continue to
assume the existence of a rate constraint~$P$, which, for mathematical
purposes may, as previously observed, be absorbed into the cost
functions $C_t(x)$ by appropriately modifying them.  We give a number
of results for this case, illustrating them with examples based on
real-world data.

It follows from the results of the previous section that, given an
initial level $S^*_0$ and a final level $S^*_T$ of the store, there
exists a pair $(S^*,\mu^*)$ as in Theorem~\ref{thm:ls} and such that
$S^*$ defines the optimal control of the store over the time interval
$[0,T]$.  One immediate consequence of this is that the optimal
control is here bang-bang is the sense that, at each time $t$, the
store should either buy as much as possible (subject to the capacity
and rate constraints), do nothing, or sell as much as possible,
according to whether the current ``reference value'' $\mu_t$ of $\mu$
is above the buy price~$c^{(b)}_t$, between $c^{(b)}_t$ and the
(lower) sell price~$c^{(s)}_t$, or below $c^{(s)}_t$.

Typically we may have $c^{(s)}_t=\eta c^{(b)}_t$ for some factor
$\eta\le1$ which may be interpreted as representing the efficiency of
the store.  As $\eta$ is decreased below $1$ the set of times at which
buying or selling actually takes place is correspondingly
reduced---see the example below.

Now note that, apart from the obvious scale factor, the solution
$(S^*,\,\mu^*)$ to the optimization problem~$\mathbf{P}$ of problem
depends on capacity constraint~$E$ and the rate constraint~$P$ only
through the ratio $E/P$, which has the dimension of \emph{time}.  As
the store capacity $E$ is increased (with $P$ held fixed), the time
horizon required for the determination of each optimal action becomes
longer and the corresponding optimal solution more global in
character.  For $E=\infty$ there is some scalar $\mu^*$ such that
$\mu^*_t=\mu^*$ for all $t$, so that in an optimal solution, at each
time $t$, the store buys if and only if $c^{(b)}_t\le\mu^*$ and sells
if and only if $c^{(s)}_t\ge\mu^*$.  The scalar $\mu^*$ is such that
the final level of the store is $S^*_T$ as required.  In contrast, as
the rate constraint $P$ is increased (with $E$ held fixed), the time
horizon required for the determination of each optimal action becomes
shorter and the corresponding optimal solution more local in
character.  These results are illustrated in the examples that follow.

We illustrate our methodology with an example storage facility using
parameters motivated by the Dinorwig pumped-storage power station in
Snowdonia, north Wales---see \cite{Din} for a good description of this
power station and its uses. (Note, however, that Dinorwig is not
currently primarily used for price arbitrage, but rather for the
provision of fast response services to the GB energy network.)  We use
the ``small store'' cost structure~\eqref{eq:1}, with $c^{(s)}_t=\eta
c^{(b)}_t$ for all $t$. A typical value for~$\eta$ would be~$0.75$
reflecting the approximate efficiency of the Dinorwig plant.  We
assume also a common input and output rate constraint $P_i=P_o=P$,
say.  The cost series are proportional to the real half-hourly spot
market wholesale electricity prices during the period corresponding to
the example.  As might be expected these prices show a strong daily
cyclical behaviour.  As already observed, for the ``small store''
essentially linear cost structure~\eqref{eq:1}, the optimal control is
bang-bang in the sense already described above.  

In the first of our examples we take the ratio $E/P$ to correspond to
10 half-hourly periods---the total length of time which the Dinorwig
facility takes to either fill or empty.  Specially, we considered the
choice~$E=10$ energy units and~$P=1$ energy unit per half-hour.
Figures~\ref{fig:A-0.65}, \ref{fig:A-0.75} and~\ref{fig:A-0.85} show
the two price series~$c^{(b)}_t$ and~$c^{(s)}_t$ for the 7-day period
Sunday 9 January 2011 to Saturday 15 January 2011 inclusive for
efficiencies of~$\eta=0.65, 0.75$ and~$0.85$, respectively.  The
decisions to buy, sell or keep the level of the store unchanged are
indicated by the red, blue and black line segments, respectively. In the lower
panel we show the series of storage values,~$S_t$, over this one-week
period. Each day storage is emptied when prices are sufficiently high
and filled when prices are low. Notice that as the efficiency,~$\eta$,
is reduced the number of periods at which it is economic to either buy
or sell (as opposed to doing nothing) is similarly reduced.
\begin{figure}[!t]
  \centering
  \includegraphics[width=3.5in]{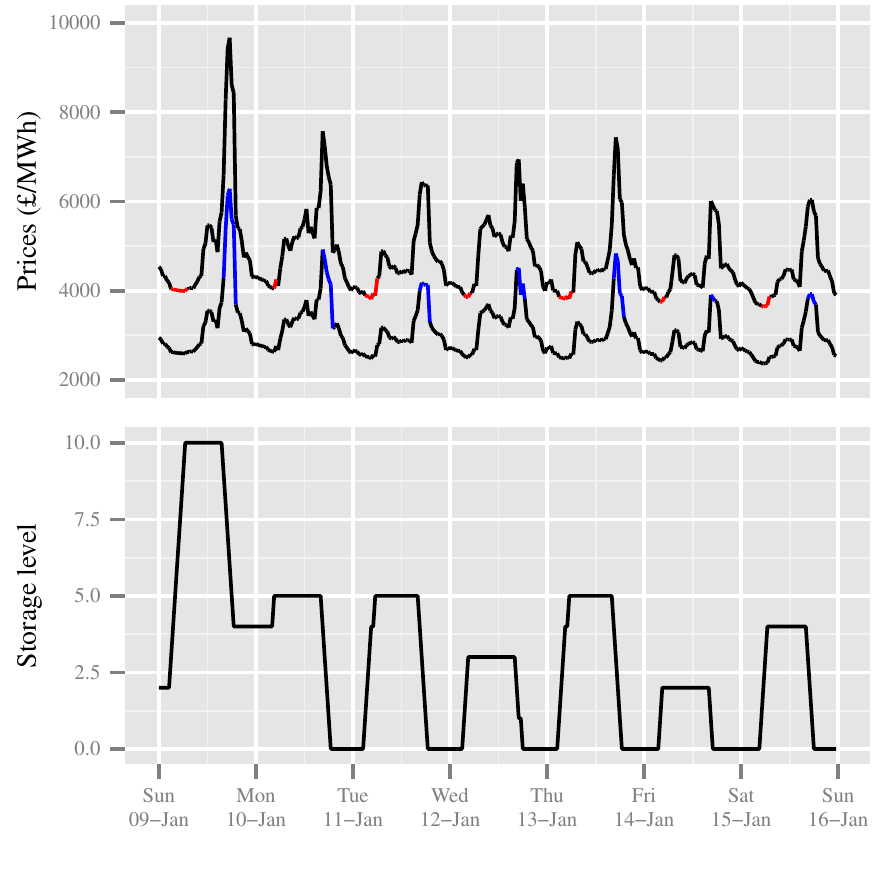}
  \caption{Example in which $E/P$ corresponds to 10 half-hourly
    periods and~$\eta=0.65$: plots of price series with buy and sell times and of
    corresponding level of storage.}
  \label{fig:A-0.65}
\end{figure}
\begin{figure}[!t]
  \centering
  \includegraphics[width=3.5in]{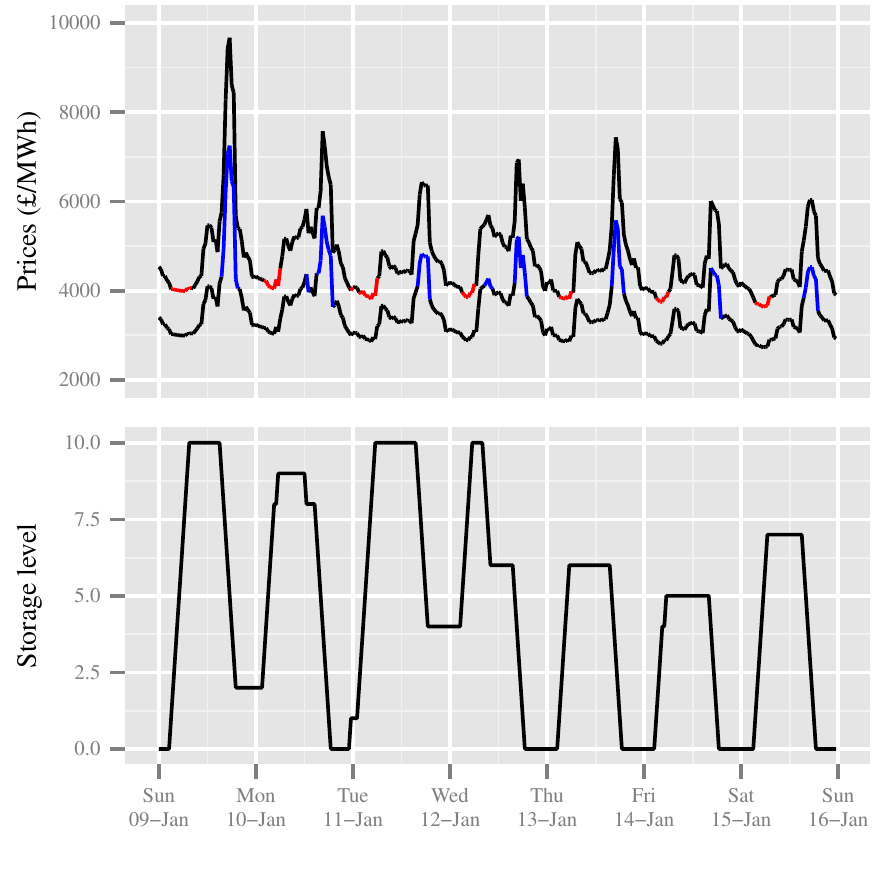}
  \caption{Example in which $E/P$ corresponds to 10 half-hourly
    periods and~$\eta=0.75$: plots of price series with buy and sell times and of
    corresponding level of storage.}
  \label{fig:A-0.75}
\end{figure}
\begin{figure}[!t]
  \centering
  \includegraphics[width=3.5in]{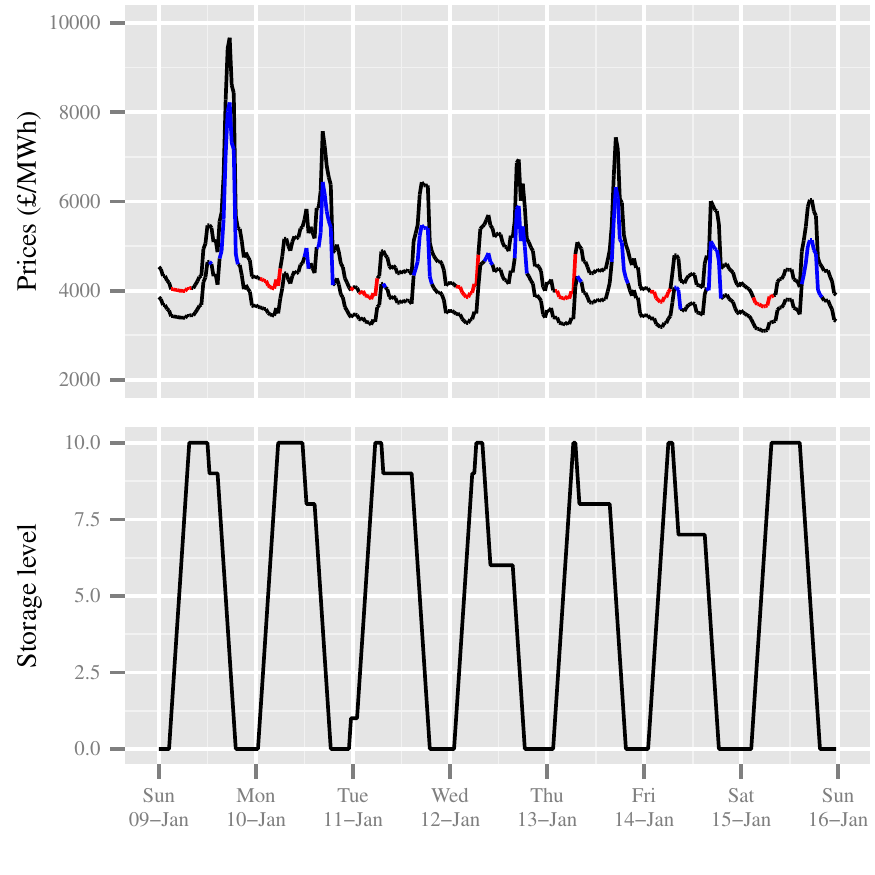}
  \caption{Example in which $E/P$ corresponds to 10 half-hourly
    periods and~$\eta=0.85$: plots of price series with buy and sell times and of
    corresponding level of storage.}
  \label{fig:A-0.85}
\end{figure}

In our second example we investigate the operation of the storage
plant with increasing storage capacity~$E$ while keeping the rate
constraint~$P$ and the two price series as before.
Figure~\ref{fig:B-Elarge} corresponds to the situation that arises
when~$E$ has increased to the extent that the capacity constraint is
no longer active provided only that the initial level~$S^*_0$ and
final level~$S^*_T$ of the store are taken sufficiently large. Here
the storage facility remains nonempty over long periods of time and
may take advantage of the price difference between, for example,
different seasons of the year.

Our third and final example shows in Figure~\ref{fig:B-Plarge} the
complementary circumstance when there is effectively no rate
constraint, that is we hold~$E$ fixed at one energy unit and
increase~$P$ until the rate constraint in no longer
active. Accordingly the finite capacity store is always able to fill
entirely and empty completely within a single half-hour period.
\begin{figure}[!t]
  \centering
  \includegraphics[width=3.5in]{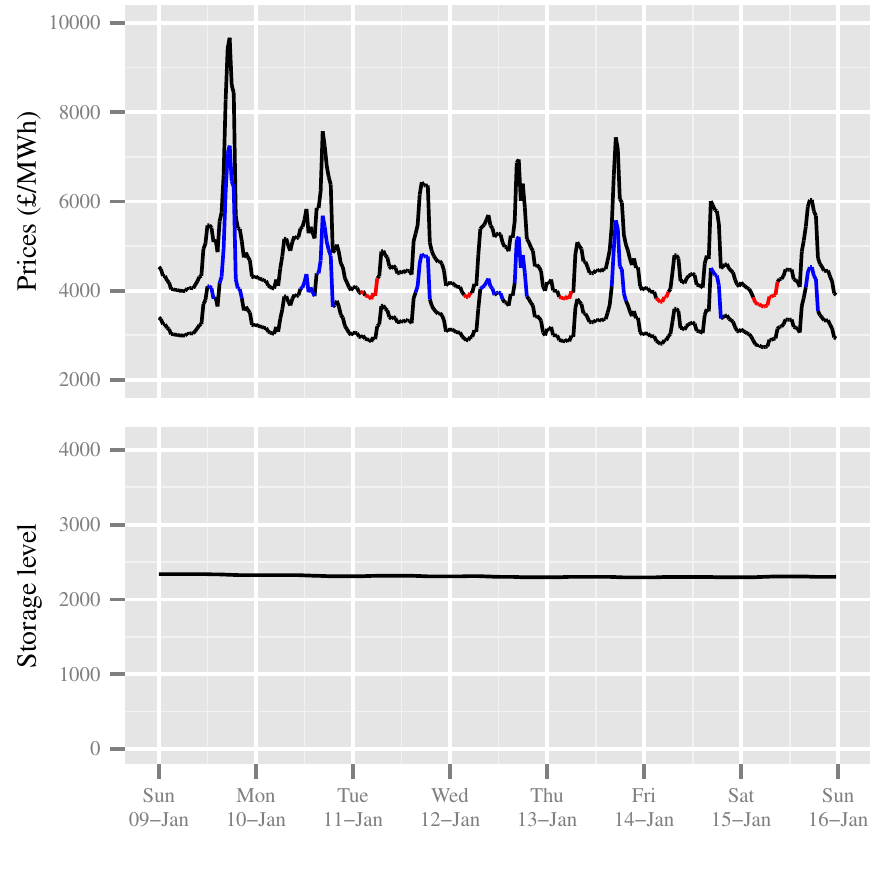}
  \caption{Example in which the capacity constraint~$E$ is no longer
    active (that is,~$E/P$ is large) and~$\eta=0.75$: plots of price
    series with buy and sell times and of corresponding level of
    storage.}
  \label{fig:B-Elarge}
\end{figure}

\begin{figure}[!t]
  \centering
  \includegraphics[width=3.5in]{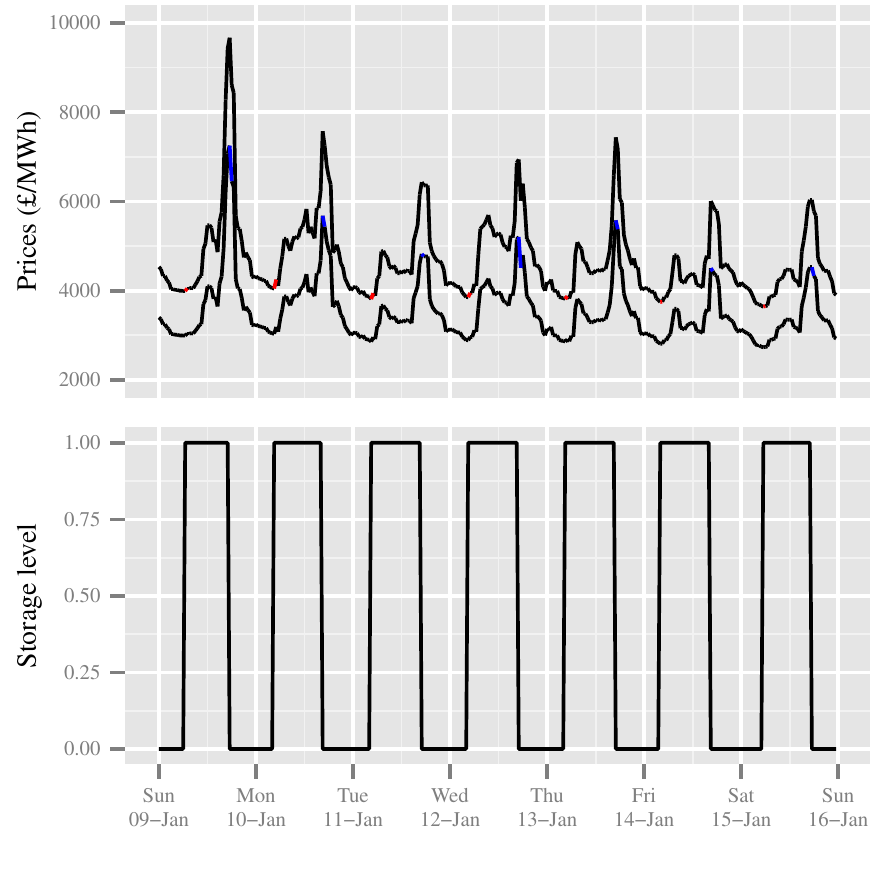}
  \caption{Example in which the rate constraint is no longer active
    (that is,~$E/P$ is small) and~$\eta=0.75$: plots of price series
    with buy and sell times and of corresponding level of storage.}
  \label{fig:B-Plarge}
\end{figure}

\section{Commentary and conclusions}
\label{sec:conclusions}

In the preceding sections we have developed the optimization theory
associated with the use of storage for arbitrage, and given an
algorithm for determining the optimal control policy for, and hence
the value of, storage when used for this purpose.  In particular our
algorithm captures the fact that the control policy is essentially
local in time, in that, for a given system subject to given capacity
and rate constraints, at each time optimal decisions are dependent
only on the relevant cost functions for what is typically a very short
time horizon.

Our model accounts for nonlinear cost functions, rate constraints,
storage inefficiencies, and the effect of externalities caused by the
activities of the store impacting the market.  What we have not done
in the present paper is to consider the use of storage for providing a
reserve in case of unexpected system shocks, such as sudden surges in
demand or shortfalls in supply.  This problem is considered by other
authors, in which the probabilities of storage underflows or overflows
are controlled to fixed levels.  However, we believe that a further
approach here would be to attach economic values to such underflows or
overflows, translating to attaching an economic worth to the absolute
level the store (as opposed to attaching a worth to a \emph{change} in
the level of the store as in the present paper).  Since in practice
storage is used both for arbitrage and for buffering or control as
described above, this would provide a more integrated approach to the
full economic valuation of such storage.

%
\IEEEpeerreviewmaketitle

\section*{Acknowledgements}
\label{sec:acknowledgements}
The authors wish to thank their co-workers Andrei Bejan, Janusz
Bialek, Chris Dent and Frank Kelly for very helpful discussions during
the preliminary part of this work.  They are also most grateful to the
Isaac Newton Institute for Mathematical Sciences in Cambridge for
their funding and hosting of a number of most useful workshops to
discuss this and other mathematical problems arising in particular in
the consideration of the management of complex energy systems.  They
are further grateful to National Grid plc for additional discussion
and the provision of data, and finally to the Engineering and Physical
Sciences Research Council for the support of the research programme
under which the present research is carried out. (The EPSRC grant
references are as follows: EP/I017054/1 and EP/I016023/1.)



%

\end{document}